\documentclass[12pt, reqno, twoside, letterpaper]{amsart}

\usepackage{paperstyle}
\usepackage{graphicx}

\usepackage{mathtools}
\usepackage{todonotes}
\usepackage[norefs,nocites]{refcheck}
\usepackage{amsmath,amsthm,amscd,amssymb}

\newcommand{\be}{\begin{equation}}
\newcommand{\ee}{\end{equation}}
\newcommand{\dalign}[1]{\[\begin{aligned} #1 \end{aligned}\]}

\newcommand{\euL}{\EuScript{L}}

\newcommand{\er}{\mathrm{e}}

\newcommand{\kzeta}{{\frac{\zeta~~}{\zeta}}^{\hskip-5pt(k)}\hskip-5pt}

\title[Riemann Hypothesis via the generalized von Mangoldt function]
{The Riemann Hypothesis via \\
the generalized von Mangoldt function}
      
\author[W.~Banks]{William Banks}

\address{Department of Mathematics, 
         University of Missouri, 
         Columbia MO 65211, USA.}

\email{bankswd@missouri.edu}
    
\author[S.~Sinha]{Saloni Sinha}

\address{Department of Mathematics, 
         University of Missouri, 
         Columbia MO 65211, USA.}

\email{ssvf4@mail.missouri.edu}
    
\date{\today}

\begin{document}

\begin{abstract}
Gonek, Graham, and Lee have shown recently that the
Riemann Hypothesis (RH) can be reformulated
in terms of certain asymptotic estimates for twisted sums
with von Mangoldt function $\Lambda$. Building on their ideas,
for each $k\in\N$, we study twisted sums with the
\emph{generalized von Mangoldt function}
$$
\Lambda_k(n)\defeq\sum_{d\,\mid\,n}\mu(d)\Big(\log\frac{n}{d}\,\Big)^k
$$
and establish similar connections with RH.
For example, for $k=2$ we show
that RH is equivalent to the assertion that,
for any fixed $\eps>0$, the estimate
$$
\sum_{n\le x}\Lambda_2(n)n^{-iy}
=\frac{2x^{1-iy}(\log x-C_0)}{(1-iy)}
-\frac{2x^{1-iy}}{(1-iy)^2}
+O\big(x^{1/2}(x+|y|)^\eps\big)
$$
holds uniformly for all $x,y\in\R$, $x\ge 2$; hence, the validity
of RH is governed by the distribution of almost-primes in the
integers. We obtain similar results for the function
$$
\Lambda^k\defeq\mathop{\underbrace{\,\Lambda\star\cdots\star\Lambda\,}}\limits_{k\text{~copies}}\,,
$$
the $k$-fold convolution of the von Mangoldt function.
\end{abstract}

\thanks{MSC Primary: 11M26; Secondary: 11M06.}


\maketitle


\tableofcontents

\newpage
\section{Introduction}

\subsection{Background}

Let $\zeta(s)$ be the Riemann zeta function. In terms of
the complex variable $s=\sigma+it$, the
\emph{Riemann Hypothesis} (RH) is the assertion that all of the
nontrivial zeros of $\zeta(s)$ lie on the line $\sigma=\tfrac12$;
in other words,
$$
\sigma>\tfrac12
\quad\Longrightarrow\quad\zeta(\sigma+it)\ne 0.
$$
In a recent paper, Gonek, Graham, and Lee~\cite{GGL} have shown
that a necessary and sufficient condition for the truth of
the Riemann Hypothesis is that for any fixed
constants $\eps,B>0$, one has the uniform estimate
\be\label{eq:GGL}
\sum_{n\le x}\Lambda(n) n^{-iy}
=\frac{x^{1-iy}}{1-iy}+O(x^{1/2}|y|^\eps)
\qquad(2\le x\le |y|^B),
\ee
where $\Lambda$ is the von Mangoldt function; see \cite[Thm.~1]{GGL}
and its proof. In the present paper, we study two distinct families of sums
that are similar to the sum $\sum_{n\le x}\Lambda(n) n^{-iy}$
above, and we show that the appropriate
 analogue of \eqref{eq:GGL} also yields a necessary and
sufficient condition for the truth of RH.
Our results suggest that, in addition to influencing the distribution
of primes, \emph{the zeros of the zeta function also exert a strong influence on
the distribution of almost-primes} and other related classes of integers.

\subsection{Statement of results}

In this paper, we study twisted sums of the form
\be\label{eq:psi-up-k}
\psi^k(x,y)\defeq\sum_{n\le x}\Lambda^k(n)n^{-iy}\qquad(k\in\N,~x,y\in\R),
\ee
where $\Lambda^k$ denotes the $k$-fold convolution of
the von Mangoldt function $\Lambda$:
$$
\Lambda^k\defeq\mathop{\underbrace{\,\Lambda\star\cdots\star\Lambda\,}}\limits_{k\text{~copies}}.
$$
We have
\be\label{eq:identity-one}
\sum_{n=1}^\infty\frac{\Lambda^k(n)}{n^s}
=(-1)^k\Big\{\frac{\zeta'}{\zeta}(s)\Big\}^k
\qquad(\sigma>1).
\ee
Note that the function $\Lambda^k$ is supported on the set 
of natural numbers that have at most $k$ distinct prime divisors,
i.e., positive integers for which $\omega(n)\le k$. 
\begin{theorem}\label{thm:main-forward}
Fix $k\in\N$. If the Riemann Hypothesis is true, then
$$
\psi^k(x,y)=(-1)^k\mathop{{\tt Res}}\limits_{w=1-iy}
\bigg(\Big\{\frac{\zeta'}{\zeta}(w+iy)\Big\}^k\frac{x^w}{w}\bigg)
+O\big(x^{1/2}\{\log(x+|y|)\}^{2k+1}\big)
$$
holds uniformly for all $x,y\in\R$, $x\ge 2$,
where the implied constant depends only on $k$.
The residual term can be omitted if $|y|>\sqrt{x}$,
and the exponent $2k+1$ can be replaced by $2$ in
the case that $k=1$.
\end{theorem}

The proof is a straightforward
application of Perron's formula coupled with bounds on 
$\zeta'(s)/\zeta(s)$. When $k=1$, we follow closely the proof
of \cite[Thm.~1]{GGL} (see also \cite[Thm.~13.1]{MontVau}). For
larger values of $k$, a slightly different approach is needed
since we do not know that the zeros of the zeta function are all simple
(this is not a problem when $k=1$). At the cost of an extra factor $\log(x+|y|)$
in the error term, we avoid this issue by shifting our line of integration
close to \emph{but not beyond} the vertical line $\sigma=\frac12$.

Theorem~\ref{thm:main-forward} (with $k\defeq 1$) yields the uniform estimate
$$
\sum_{n\le x}\Lambda(n)n^{-iy}=\frac{x^{1-iy}}{1-iy}
+O\big(x^{1/2}\{\log(x+|y|)\}^2\big)\qquad(x,y\in\R,~x\ge 2)
$$
under RH. This result strengthens the estimate \eqref{eq:GGL}
obtained in~\cite{GGL}. Moreover, taking $y\defeq 0$,
we recover the well known result of von Koch \cite{vonKoch} which
asserts that, under RH, one has
$$
\psi(x)=x+O\big(x^{1/2}(\log x)^2\big)\qquad(x\ge 2),
$$
where $\psi$ is the Chebyshev function.
As another example, Theorem~\ref{thm:main-forward} (with $k\defeq 2$) provides
the conditional estimate
$$
\sum_{n\le x}(\Lambda\star\Lambda)(n)n^{-iy}
=\frac{x^{1-i y}(\log x-2C_0)}{1-iy}
-\frac{x^{1-i y}}{(1-iy)^2}
+O\big(x^{1/2}\{\log(x+|y|)\}^5\big)
$$
holds, where $C_0$ is the Euler-Mascheroni constant. In particular, under RH
we have
$$
\sum_{n\le x}(\Lambda\star\Lambda)(n)
=x(\log x-2C_0-1)+O\big(x^{1/2}(\log x)^5\big).
$$

Our next result is the following strong converse of Theorem~\ref{thm:main-forward}.

\begin{theorem}\label{thm:main-back}
Fix $k\in\N$, and suppose that for any $\eps>0$ the estimate
$$
\psi^k(x,y)=(-1)^k\mathop{{\tt Res}}\limits_{w=1-iy}
\bigg(\Big\{\frac{\zeta'}{\zeta}(w+iy)\Big\}^k\frac{x^w}{w}\bigg)
+O\big(x^{1/2}(x+|y|)^\eps\big)
$$
holds uniformly for all $x,y\in\R$, $x\ge 2$,
where the implied constant depends only on $k$ and $\eps$.
Then the Riemann Hypothesis is true.
\end{theorem}

Our proof is an adaptation of the second half of the proof of \cite[Thm.~1]{GGL}

\bigskip
We also consider twisted sums with the \emph{generalized von Mangoldt
function} (see, for example, \cite[\S1.4]{IwanKow}). More specifically, 
we study sums of the form
$$
\psi_k(x,y)\defeq\sum_{n\le x}\Lambda_k(n)n^{-iy}\qquad(k\in\N,~x,y\in\R),
$$
where $\Lambda_k\defeq \mu\star L^k$ with $\mu$ the M\"obius function
and $L$ the natural logarithm; in other words,
\be\label{eq:gen-vonM-defn}
\Lambda_k(n)\defeq\sum_{d\,\mid\,n}\mu(d)\Big(\log\frac{n}{d}\,\Big)^k.
\ee
Note that
\be\label{eq:identity-two}
\sum_{n=1}^\infty\frac{\Lambda_k(n)}{n^s}
=(-1)^k\kzeta(s)\qquad(\sigma>1).
\ee
The function $\Lambda_k$ is supported on the set 
of natural numbers that have no more than $k$ distinct prime divisors.

\begin{theorem}\label{thm:main2-forward}
Fix $k\in\N$. If the Riemann Hypothesis is true, then the estimate
$$
\psi_k(x,y)=(-1)^k\mathop{{\tt Res}}\limits_{w=1-iy}
\bigg(\kzeta(w+iy)\frac{x^w}{w}\bigg)
+O\big(x^{1/2}\{\log(x+|y|)\}^{2k+1}\big)
$$
holds uniformly for all $x,y\in\R$, $x\ge 2$,
where the implied constant depends only on $k$.
The residual term can be omitted if $|y|>\sqrt{x}$,
and the exponent $2k+1$ can be replaced by $2$ in
the case that $k=1$.
\end{theorem}

For example, Theorem~\ref{thm:main2-forward} (with $k\defeq 2$) asserts
that the conditional estimate
$$
\sum_{n\le x}\Lambda_2(n)n^{-iy}
=\frac{2x^{1-iy}(\log x-C_0)}{(1-iy)}
-\frac{2x^{1-iy}}{(1-iy)^2}
+O\big(x^{1/2}\{\log(x+|y|)\}^5\big)
$$
holds uniformly for all $x,y\in\R$, $x\ge 2$. In particular, under RH we have
$$
\sum_{n\le x}\Lambda_2(n)
=2x(\log x-C_0-1)+O\big(x^{1/2}(\log x)^5\big).
$$

\begin{theorem}\label{thm:main2-back}
Fix $k\in\N$, and suppose that for any $\eps>0$ the estimate
$$
\psi_k(x,y)=(-1)^k\mathop{{\tt Res}}\limits_{w=1-iy}
\bigg(\kzeta(w+iy)\frac{x^w}{w}\bigg)
+O\big(x^{1/2}(x+|y|)^\eps\big)
$$
holds uniformly for all $x,y\in\R$, $x\ge 2$,
where the implied constant depends only on $k$ and $\eps$.
Then the Riemann Hypothesis is true.
\end{theorem}

We expect that similar results can be achieved for a much
wider class of arithmetic functions.

\section{Some auxiliary results}

In this section, we present a series of technical lemmas that are used below in 
the proofs of our theorems.

\begin{lemma}\label{lem:vonM-bounds}
For all $k,n\in\N$ we have
$$
0\le\Lambda^k(n)\le\Lambda_k(n)\le(\log n)^k.
$$
\end{lemma}

\begin{proof}
Let $L$ be the natural logarithm function, and define
$\Lambda^0$ and $\Lambda_0$ to be the indicator function of the number one.
Since $\Lambda$ takes only nonnegative values, it follows that
$\Lambda^k\ge 0$, and the upper bound $\Lambda_k\le L^k$ is well known;
see, for example, \cite[(1.45)]{IwanKow}.
It remains to show that $\Lambda^k\le\Lambda_k$.
Using the identity $\Lambda_{m+1}=L\Lambda_m+\Lambda\star\Lambda_m$
(see \cite[(1.44)]{IwanKow}) we have
$$
\Lambda^{\ell-1}\star\Lambda_{k-\ell+1}-\Lambda^\ell\star\Lambda_{k-\ell}
=\Lambda^{\ell-1}\star(L\Lambda_{k-\ell})\ge 0.
$$
Summing this bound with $\ell=1,\ldots,k$, we see that
$\Lambda_k\ge\Lambda^k$.
\end{proof}

\begin{lemma}\label{lem:saloni}
For all $m\in\N$, $y\in\R$, and $s\in\C$ with $\sigma>2$, we have
$$
\int_1^\infty\mathop{{\tt Res}}\limits_{w=1}
\bigg(\frac{1}{(w-1)^m}
\frac{x^{w-iy}}{w-iy}\bigg)x^{-s}\,dx
=\frac{1}{s-1}\(\frac{1}{(s-2+iy)^m}-\frac{1}{(iy-1)^m}\).
$$
\end{lemma}

\begin{proof}
We compute
\begin{align}
\nonumber
&\mathop{{\tt Res}}\limits_{w=1}
\bigg(\frac{1}{(w-1)^m}
\frac{x^{w-iy}}{w-iy}\bigg)
=\frac{1}{(m-1)!}\lim_{w\to 1}\frac{d^{m-1}}{dw^{m-1}}
\bigg(\frac{x^{w-iy}}{w-iy}\bigg)\\
\nonumber
&\qquad\qquad=\frac{1}{(m-1)!}\lim_{w\to 1}
\sum_{j=0}^{m-1}{m-1\choose j}
x^{w-iy}(\log x)^{m-j-1}\frac{(-1)^jj!}{(w-iy)^{j+1}}\\
\label{eq:residue-reln}
&\qquad\qquad=-\sum_{j=0}^{m-1}
\frac{x^{1-iy}(\log x)^{m-1-j}}{(m-1-j)!(iy-1)^{j+1}}.
\end{align}
Using the identity
$$
\int_1^\infty x^{1-iy}(\log x)^{m-1-j}\cdot x^{-s}\,dx
=\frac{(m-j-1)!}{(s-2+iy)^{m-j}},
$$
we get that
\dalign{
&\int_1^\infty\mathop{{\tt Res}}\limits_{w=1}
\bigg(\frac{1}{(w-1)^m}
\frac{x^{w-iy}}{w-iy}\bigg)x^{-s}\,dx
=-\sum_{j=0}^{m-1}\frac{1}{(iy-1)^{j+1}(s-2+iy)^{m-j}}\\
&\qquad=\frac{1}{s-1}
\sum_{j=0}^{m-1}\bigg(\frac{1}{(iy-1)^j(s-2+iy)^{m-j}}
-\frac{1}{(iy-1)^{j+1}(s-2+iy)^{m-j-1}}\bigg).
}
Evaluating the telescoping sum, the lemma follows.
\end{proof}		

\begin{lemma}\label{lem:residual}
For all $k\in\N$ and $x,y\in\R$ with $x\ge 1$, we have
$$
\mathop{{\tt Res}}\limits_{w=1-iy}
\bigg(\Big\{\frac{\zeta'}{\zeta}(w+iy)\Big\}^k\frac{x^w}{w}\bigg)
\ll\frac{x(\log x)^{k-1}}{|y|+1}
$$
and
$$
\mathop{{\tt Res}}\limits_{w=1-iy}
\bigg(\kzeta(w+iy)\frac{x^w}{w}\bigg)
\ll\frac{x(\log x)^{k-1}}{|y|+1},
$$
where the implied constants depend only on $k$.
\end{lemma}

\begin{proof}
Using the Laurent series development of $\{\zeta'(s)/\zeta(s)\}^k$
at $s=1$, we write
$$
\Big\{\frac{\zeta'}{\zeta}(s)\Big\}^k=f(s)+g(s),\qquad
f(s)\defeq\sum_{m=1}^k\frac{a_m}{(s-1)^m},
$$
where $g$ is analytic in a neighborhood of one.
Therefore,
\dalign{
\mathop{{\tt Res}}\limits_{w=1-iy}
\bigg(\Big\{\frac{\zeta'}{\zeta}(w+iy)\Big\}^k\frac{x^w}{w}\bigg)
&=\mathop{{\tt Res}}\limits_{w=1}
\bigg(\Big\{\frac{\zeta'}{\zeta}(w)\Big\}^k\frac{x^{w-iy}}{w-iy}\bigg)\\
&=\sum_{m=1}^k a_m\mathop{{\tt Res}}\limits_{w=1}
\bigg(\frac{1}{(w-1)^m}\frac{x^{w-iy}}{w-iy}\bigg).}
For each $m=1,\ldots,k$, we have by \eqref{eq:residue-reln}:
$$
\mathop{{\tt Res}}\limits_{w=1}
\bigg(\frac{1}{(w-1)^m}
\frac{x^{w-iy}}{w-iy}\bigg)
\ll\frac{x(\log x)^{k-1}}{|y|+1},
$$
and so we obtain the first bound of the lemma. The proof of the second
bound is similar.
\end{proof}
	
\section{Bounds on $\zeta^{(k)}/\zeta$ and $(\zeta'/\zeta)^k$} 

In this section, any implied constants in the symbols $O$ and $\ll$
may depend (where obvious) on the parameters $k$, $\ell$, and $\delta$, 
but are absolute otherwise. For any complex number $s=\sigma+it$,
we denote $\tau\defeq|t|+2$.

\begin{lemma}\label{lem:log-deriv-bds-uncondl}
For any real number $T_1\ge 2$, there exists $T\in[T_1,T_1+1]$ such that
\be\label{eq:LD-horiz}
\frac{\zeta'}{\zeta}(\sigma+iT)\ll(\log T)^2\qquad(-1\le\sigma\le 2).
\ee
We also have
\be\label{eq:LD-vert}
\frac{\zeta'}{\zeta}(-1+it)\ll\log\tau\qquad(t\in\R).
\ee
\end{lemma}
\begin{proof}
See \cite[Lemma~12.2]{MontVau} and \cite[Lemma~12.4]{MontVau}.
\end{proof}

In addition to Lemma~\ref{lem:log-deriv-bds-uncondl}
we use a \emph{conditional} bound on the logarithmic derivative of $\zeta(s)$;
see Proposition~\ref{prop:condl-bd-log-deriv-zeta} below. We achieve this via
the following technical lemmas.

\begin{lemma}\label{lem:summing-roots}
Assume RH. For any $x\ge 2$, the bounds
\be\label{eq:sum-over-roots-1}
\sum_\rho\bigg(\frac{1}{s-\rho}+\frac{1}{\rho}\bigg)
\ll\log (x\tau)\log\tau
\ee
and
\be\label{eq:sum-over-roots-k}
\sum_\rho\frac{1}{(s-\rho)^k}
\ll(\log x)^k\log\tau\qquad(k\ge 2)
\ee
hold uniformly throughout the region
$$
\cR_x\defeq\big\{s\in\C:
\tfrac12+\tfrac{1}{\log x}\le\sigma\le 2,~|s-1|>\tfrac{1}{100}\big\},
$$
where the sums in \eqref{eq:sum-over-roots-1}
and \eqref{eq:sum-over-roots-k} run over the zeros
of $\zeta(s)$ in the critical strip.
\end{lemma}

\begin{proof}
For every integer $n$, let $\cZ_n$ be the multiset consisting of
zeros $\rho=\frac12+i\gamma$ of the zeta function that satisfy the
equivalent conditions
$$
n\le t-\gamma\le n+1
\qquad\Longleftrightarrow\qquad
t-n-1\le\gamma\le t-n,
$$
where each zero $\rho$ in $\cZ_n$ appears a number of times that is
equal to its multiplicity. Note that
\be\label{eq:s-minus-rho}
|s-\rho|=\big|\sigma-\tfrac12+i(t-\gamma)\big|
\ge\max\{\tfrac{1}{\log x},|n|-1\}\qquad(s\in\cR_x,~\rho\in\cZ_n).
\ee
Since $\zeta(s)$ has no zeros in the critical strip with
$\gamma\in[-14,14]$, it follows that $\cZ_n=\varnothing$
unless $|t-n|>10$ (say). Thus, denoting
$$
\tau_n\defeq|t-n|+2,
$$
we have
\be\label{eq:just-rho}
|\rho|=\big|\tfrac12+i\gamma\big|\asymp|\gamma|\asymp\tau_n\qquad(\rho\in\cZ_n).
\ee
and (cf.~\cite[Thm.\ 10.13]{MontVau})
\be\label{eq:bonnie}
|\cZ_n|\ll\log\tau_n.
\ee

Let $s\in\cR_x$ be fixed in what follows, and define
$$
S_{k,n}\defeq\sum_{\rho\in\cZ_n}f_k(\rho)\qquad(k\in\N,~n\in\Z),
$$
where
$$
f_1(\rho)\defeq\bigg|\frac{1}{s-\rho}+\frac{1}{\rho}\bigg|
\mand
f_k(\rho)\defeq\bigg|\frac{1}{(s-\rho)^k}\bigg|\qquad(k\ge 2).
$$
and observe that
$$
\sum_\rho\bigg(\frac{1}{s-\rho}+\frac{1}{\rho}\bigg)
\ll\ssum{n\in\Z}S_{1,n}
\mand
\sum_\rho\frac{1}{(s-\rho)^k}\ll\ssum{n\in\Z}S_{k,n}
\qquad(k\ge 2).
$$

If $|n|\le 10$, then by \eqref{eq:s-minus-rho} and \eqref{eq:just-rho}:
$$
f_1(\rho)\le\frac{1}{|s-\rho|}+\frac{1}{|\rho|}
\ll \log x \qquad(\rho\in\cZ_n),
$$
and 
$$
f_k(\rho)=\frac{1}{|s-\rho|^k}\ll(\log x)^k\qquad(k\ge 2,~\rho\in\cZ_n),
$$
Using \eqref{eq:bonnie} we get that
\be\label{eq:small-n}
\sum_{|n|\le 10}S_{k,n}\ll(\log x)^k\sum_{|n|\le 10}\log\tau_n
\ll(\log x)^k\log\tau\qquad(k\in\N).
\ee

Next, suppose that $|n|>10$. For $k\ge 2$, using
\eqref{eq:s-minus-rho} and \eqref{eq:bonnie} we have
\dalign{
f_k(\rho)=\frac{1}{|s-\rho|^k}\ll\frac{1}{|n|^k}
&\qquad\Longrightarrow\qquad
S_{k,n}\ll\sum_{\rho\in\cZ_n}\frac{1}{|n|^k}\ll\frac{\log\tau_n}{|n|^k}\\
&\qquad\Longrightarrow\qquad
\sum_{|n|>10}S_{k,n}
\ll\log\tau.
}
Combined with \eqref{eq:small-n}, this completes the proof
of \eqref{eq:sum-over-roots-k}. When $k=1$, we have by
\eqref{eq:s-minus-rho}, \eqref{eq:just-rho}, and \eqref{eq:bonnie}:
$$
f_1(\rho)=\bigg|\frac{s}{(s-\rho)\rho}\bigg|
\ll\frac{\tau}{|n|\tau_n}
\qquad\Longrightarrow\qquad
S_{1,n}\ll \sum_{\rho\in\cZ_n}\frac{\tau}{|n|\tau_n}
\ll\frac{\tau\log\tau_n}{|n|\tau_n}.
$$
Recalling the definition of $\tau_n$, we see that 
$$
\tau_n\asymp\begin{cases}
\tau&\quad\hbox{if $10<|n|\le\tfrac12|t|$},\\
|t-n|+2&\quad\hbox{if $\tfrac12|t|<|n|\le 2|t|$},\\
|n|&\quad\hbox{if $|n|>2|t|$}.\\
\end{cases}
$$
In addition to this, $\tau_n\ll\tau$ in the second range above,
and $|n|\gg\tau$ in the last two ranges. Putting everything
together, we have
\dalign{
\sum_{|n|>10}S_{1,n}
&\ll\sum_{10<|n|\le\frac12|t|}\frac{\log\tau}{|n|}
+\sum_{\frac12|t|<|n|\le 2|t|}\frac{\log\tau}{|t-n|+2}
+\sum_{|n|>2|t|}\frac{\tau\log|n|}{|n|^2}
\ll(\log\tau)^2.
}
Combined with \eqref{eq:small-n}, this completes the proof
of \eqref{eq:sum-over-roots-1}.
\end{proof}

\bigskip

Let $f$ be a twice-differentiable function on $(0,\infty)$ such that
\begin{itemize}
\item[$(i)$] $f(u)\to 0$ and $f'(u)\to 0$ as $u\to\infty$;
\item[$(ii)$] $f''$ is integrable on $(1,\infty)$.
\end{itemize}
The relation
\be\label{eq:RSint}
\sum_{n=1}^\infty f'(n)=-f(1)+f'(1)+\int_1^\infty f''(u)\{u\}\,du,
\ee
which is immediate using Riemann-Stieltjes integration, plays
an important role in the proof of the following lemma.

\begin{lemma}\label{lem:gamma-deriv-estimates}
Let $\delta>0$ be fixed. The estimates
\be\label{eq:log-gamma-1}
\frac{\Gamma'}{\Gamma}(s)=\log s+O(\tau^{-1})
\ee
and
\be\label{eq:log-gamma-ell}
\frac{d^\ell}{ds^\ell}\frac{\Gamma'}{\Gamma}(s)=
(-1)^{\ell-1}(\ell-1)!\,s^{-\ell}+O(\tau^{-\ell-1})\qquad(\ell\ge 1)
\ee
hold uniformly throughout the half-plane $\{\sigma\ge\delta\}$.
\end{lemma}

\begin{proof}
The gamma function $\Gamma$ can be defined as the reciprocal of the Weierstrauss
product (see, e.g., \cite[Chap.~XII]{WW}):
$$
\Gamma(s)\defeq s^{-1}\er^{-\gamma s}
\prod_{n=1}^\infty\bigg\{\Big(1+\frac{s}{n}\Big)^{-1}\er^{s/n}\bigg\},
$$
which implies that
\be\label{eq:log-deriv-gamma}
\frac{\Gamma'}{\Gamma}(s)=-\frac{1}{s}-\gamma
+s\sum_{n=1}^\infty\frac{1}{n(n+s)}.
\ee
Using \eqref{eq:RSint} with
$$
f(u)\defeq\log\Big(\frac{u}{u+s}\Big),
$$
we have
\dalign{
s\sum_{n=1}^\infty\frac{1}{n(n+s)}
&=\log(s+1)+1-\frac{1}{s+1}-
\int_1^\infty\{u\}\bigg(\frac{1}{u^2}-\frac{1}{(u+s)^2}\bigg)\,du\\
&=\log s+\gamma+O(\tau^{-1}).
}
Combining this with \eqref{eq:log-deriv-gamma}, the first statement of
the lemma is proved.

For any positive integer $\ell$, from \eqref{eq:log-deriv-gamma} we deduce that
\be\label{eq:diff-log-deriv-gamma}
\frac{d^\ell}{ds^\ell}\frac{\Gamma'}{\Gamma}(s)
=(-1)^{\ell-1}\ell!\bigg\{\frac{1}{s^{\ell+1}}
+\sum_{n=1}^\infty\frac{1}{(n+s)^{\ell+1}}\bigg\}.
\ee
Applying \eqref{eq:RSint} with
$$
f(u)\defeq-\frac{1}{\ell(u+s)^\ell},
$$
we have
\dalign{
\sum_{n=1}^\infty\frac{1}{(n+s)^{\ell+1}}
&=\frac{1}{\ell(s+1)^\ell}+\frac{1}{(s+1)^{\ell+1}}
-(\ell+1)\int_1^\infty \frac{\{u\}\,du}{(u+s)^{\ell+2}}\\
&=\frac{1}{\ell(s+1)^\ell}+O(\tau^{-\ell-1}).
}
Combining this with \eqref{eq:diff-log-deriv-gamma}, we obtain
the second statement of the lemma.
\end{proof}

\begin{proposition}\label{prop:condl-bd-log-deriv-zeta}
Assume RH. For any $k\in\N$ and $x\ge 2$, the bounds
\be\label{eq:aloe}
\Big\{\frac{\zeta'}{\zeta}(s)\Big\}^k\ll\big(\log(x\tau)\log\tau\big)^k
\ee
and
\be\label{eq:vera}
\kzeta(s)\ll\big(\log(x\tau)\log\tau\big)^k
\ee
hold uniformly throughout the region
$$
\cR_x\defeq\big\{s\in\C:
\tfrac12+\tfrac{1}{\log x}\le\sigma\le 2,~|s-1|>\tfrac{1}{100}\big\}.
$$
\end{proposition}

\begin{proof}
It suffices to prove \eqref{eq:vera}, since \eqref{eq:aloe} follows from
case $k=1$ of \eqref{eq:vera}.

Let $z_j\defeq\zeta^{(j)}/\zeta$ for each $j\in\N$.
Our goal is to show that
\be\label{eq:zeta-j-over-zeta}
z_j\ll\big(\log(x\tau)\log\tau\big)^j\qquad(s\in\cR_x)
\ee
holds for all $j$, where the implied constant depends only on $j$.

For $j=1$ we have (cf.~\cite[(10.29)]{MontVau})
\be\label{eq:bear}
z_1(s)=B+\frac12\log\pi-\frac{1}{s-1}
-\frac12\frac{\Gamma'}{\Gamma}(s/2+1)+
\sum_\rho\bigg(\frac{1}{s-\rho}+\frac{1}{\rho}\bigg),
\ee 
where $B$ is a constant, and
the sum runs over the zeros of $\zeta(s)$ in the critical strip.
The case $j=1$ of \eqref{eq:zeta-j-over-zeta} follows from \eqref{eq:bear}
and Lemmas~\ref{lem:summing-roots} and~\ref{lem:gamma-deriv-estimates}.

Now suppose \eqref{eq:zeta-j-over-zeta} holds for all positive integers $j<k$, where $k\ge 2$.
On the one hand, differentiating both
sides of \eqref{eq:bear} precisely $k-1$ times, we have
$$
z_1^{(k-1)}(s)=\frac{(-1)^k(k-1)!}{(s-1)^k}
-\frac{1}{2^k}\frac{d^{k-1}}{ds^{k-1}}\frac{\Gamma'}{\Gamma}(s/2+1)+
\sum_\rho\frac{(-1)^{k-1}(k-1)!}{(s-\rho)^k}.
$$
Using Lemmas~\ref{lem:summing-roots} and~\ref{lem:gamma-deriv-estimates}
again, this implies that the bound
\be\label{eq:picnic}
z_1^{(k-1)}(s)\ll (\log x)^k\log\tau
\ee
holds uniformly in $\cR_x$. On the other hand, by the quotient rule, one has
the simple relation
$$
z'_m=z_{m+1}-z_1z_m\qquad(m\in\N).
$$
Using an inductive argument, one can show that 
$$
z_m^{(n)}=z_{m+n}+P_n(z_1,\ldots,z_{m+n-1})\qquad(m,n\in\N),
$$
where each $P_n$ is a polynomial in $\Z[X_1,\ldots,X_{m+n-1}]$ whose monomials
all have the form $\prod_jX_j^{n_j}$ for some integers $n_j\ge 0$
that satisfy $\sum_j jn_j=m+n$. In particular, we have
$$
z_1^{(k-1)}-z_k=P_{k-1}(z_1,\ldots,z_{k-1}),
$$
and each monomial term occurring in the polynomial on the right side satisfies
the uniform bound
$$
\prod_{j=1}^{k-1}z_j(s)^{n_j}\ll\prod_{j=1}^{k-1}\big(\log(x\tau)\log\tau\big)^{jn_j}
=\big(\log(x\tau)\log\tau\big)^k\qquad(s\in\cR_x).
$$
In other words,
$$
z_1^{(k-1)}(s)-z_k(s)\ll\big(\log(x\tau)\log\tau\big)^k\qquad(s\in\cR_x).
$$
Combining this bound with \eqref{eq:picnic}, it follows that \eqref{eq:zeta-j-over-zeta}
holds when $j=k$. This completes the induction, and the proposition is proved.
\end{proof}

\section{Proof of Theorems~\ref{thm:main-forward} and~\ref{thm:main2-forward}}

Both theorems are proved in parallel.

Let the integer $k\in\N$ be fixed throughout.
Below, any implied constants in the symbols $O$ and $\ll$
may depend (where obvious) on $k$ but are independent of other parameters.

For the proof of Theorem~\ref{thm:main-forward}, we set
$$
a_n(y)\defeq \Lambda^k(n)n^{-iy}
\mand
\alpha(y,s)\defeq\sum_{n=1}^\infty\frac{a_n(y)}{n^s}
=(-1)^k\Big\{\frac{\zeta'}{\zeta}(s+iy)\Big\}^k.
$$
Note that
$$
\sum_{n\le x}a_n(y)=\psi^k(x,y).
$$
On the other hand, for the proof of Theorem~\ref{thm:main2-forward}, we put
$$
a_n(y)\defeq \Lambda_k(n)n^{-iy}
\mand
\alpha(y,s)\defeq\sum_{n=1}^\infty\frac{a_n(y)}{n^s}
=(-1)^k\,\kzeta(s+iy),
$$
and we have
$$
\sum_{n\le x}a_n(y)=\psi_k(x,y).
$$
In either case, our aim is to show that
$$
\sum_{n\le x}a_n(y)=\mathop{{\tt Res}}\limits_{w=1-iy}
\Big(\alpha(y,w)\frac{x^w}{w}\Big)
+O\big(x^{1/2}\{\log(x+|y|)\}^{2k+1}\big),
$$
where the exponent $2k+1$ can be replaced by $2$ in the case that $k=1$.
Note that Lemma~\ref{lem:residual} guarantees that
the residual term can be omitted when $|y|>\sqrt{x}$.

Let $x,y\in\R$ with $x\ge 2$. Let
$$
\sigma_0\defeq 1+1/\log x
\mand
T\in\big[\sqrt{x}+10,\sqrt{x}+11\big],
$$
where $T$ is chosen so that \eqref{eq:LD-vert} holds.
Adjusting $T$ \emph{slightly}, we can assume that
$T$ is not the ordinate of any zero of the zeta function.
By Perron's formula (see 
\cite[Thm.~5.2 and Cor.~5.3]{MontVau}) we have
\be\label{eq:aoc}
\sum_{n\le x}a_n(y)
=\frac{1}{2\pi i}\int_{\sigma_0-iT}^{\sigma_0+iT}
\alpha(y,s)\,\frac{x^s}{s}\,ds+O(E_1+E_2+E_3),
\ee
where the error terms are given by
$$
E_1\defeq\sum_{x/2<n<2x}a_n(0)\min\Big\{1,\frac{x}{T|x-n|}\Big\},
\qquad
E_2\defeq\frac{x}{T}\sum_{n=1}^\infty\frac{a_n(0)}{n^{\sigma_0}},
$$
and
$$
E_3\defeq\begin{cases}
a_x(0)&\quad\hbox{if $x\in\N$,}\\
0&\quad\hbox{otherwise}.
\end{cases}
$$
Since $a_n(0)=\Lambda^k(n)$ or $\Lambda_k(n)$,
Lemma~\ref{lem:vonM-bounds} shows that
\be\label{eq:terms13}
E_1\ll\frac{x(\log x)^{k+1}}{T}\le x^{1/2}(\log x)^{k+1}
\mand
E_3\le (\log x)^k.
\ee
Moreover, as
$$
\bigg|\sum_{n=1}^\infty\frac{a_n(0)}{n^{\sigma_0}}\bigg|
=\big|\alpha(0,\sigma_0)\big|=
\bigg|\frac{\zeta'}{\zeta}(\sigma_0)\bigg|^k
\quad\text{or}\quad
\bigg|\kzeta(\sigma_0)\bigg|,
$$
using the Laurent series development of $(\zeta'(s)/\zeta(s))^k$ or 
$\zeta^{(k)}(s)/\zeta(s)$
at $s=1$, we derive the bound
\be\label{eq:term2}
E_2\ll\frac{x(\log x)^k}{T}\ll x^{1/2}(\log x)^k.
\ee
Combining \eqref{eq:terms13} and \eqref{eq:term2} with \eqref{eq:aoc},
we derive the estimate
\be\label{eq:rush}
\sum_{n\le x}a_n(y)
=\frac{1}{2\pi i}\int_{\sigma_0-iT}^{\sigma_0+iT}
\alpha(y,s)\,\frac{x^s}{s}\,ds
+O\big(x^{1/2}(\log x)^{k+1}\big).
\ee

Next, we shift the line of integration of the integral in \eqref{eq:rush}.
For this, we study the cases $k=1$ and $k\ge 2$, separately.

\bigskip\noindent{\sc Case 1 $(k=1)$.}
In this case, for both Theorems~\ref{thm:main-forward} and~\ref{thm:main2-forward}
we have
$$
\alpha(y,s)=-\frac{\zeta'}{\zeta}(s+iy).
$$
Let $\sC$ be the rectangular contour in $\C$ that connects
$$
\sigma_0-iT
~~\longrightarrow~~\sigma_0+iT
~~\longrightarrow~~ -1+iT
~~\longrightarrow~~ -1-iT
~~\longrightarrow~~ \sigma_0-iT.
$$
Along the horizontal segment $s=\sigma+iT$ with
$\sigma\in(-1,\sigma_0)$, by \eqref{eq:LD-horiz} we see that
$$
\alpha(y,s)\ll\euL^2,
\qquad\text{where}\quad
\euL\defeq\log(x+|y|);
$$
consequently,
\be\label{eq:horiz-up}
\int_{-1+iT}^{\sigma_0+iT}
\alpha(y,s)\,\frac{x^s}{s}\,ds
\ll \frac{\euL^2}{T}\int_{-1}^{\sigma_0}x^\sigma\,d\sigma
<\frac{x^{\sigma_0}\euL^2}{T\log x}
\ll x^{1/2}\euL^2.
\ee

Similarly,
\be\label{eq:horiz-down}
\int_{-1-iT}^{\sigma_0-iT}
\alpha(y,s)\,\frac{x^s}{s}\,ds
\ll x^{1/2}\euL^2.
\ee
On the other hand, using \eqref{eq:LD-vert} we have $\alpha(y,s)\ll\euL$
along the vertical segment $s=-1+it$ with $|t|<T$; thus,
\begin{equation}
\label{eq:vert-left}
\int_{-1-iT}^{-1+iT}
\alpha(y,s)\,\frac{x^s}{s}\,ds
\ll x^{-1}\euL\int_{-T}^T\frac{dt}{|t|+1}
\ll x^{-1}\euL^2.
\end{equation}
Combining \eqref{eq:horiz-up}, \eqref{eq:horiz-down}, and
\eqref{eq:vert-left} with our previous estimate \eqref{eq:rush},
we have
\be\label{eq:rush2}
\sum_{n\le x}a_n(y)
=\frac{1}{2\pi i}
\oint_{\sC}\alpha(y,s)\,
\frac{x^s}{s}\,ds+O(x^{1/2}\euL^{2}).
\ee
Since $k=1$, the poles of the integrand inside the contour are all simple.
If $|y|<T$, there is a pole at $s=1-iy$ with residue
$$
-\mathop{{\tt Res}}\limits_{w=1-iy}
\bigg(\frac{\zeta'}{\zeta}(w+iy)\frac{x^w}{w}\bigg)=\frac{x^{1-iy}}{1-iy}.
$$
The pole at $s=0$ contributes has residue $-\zeta'(iy)/\zeta(iy)\ll\euL$,
which can be thrown into the error term
(for the bound, see \cite[Lemma~12.1]{MontVau}). The remaining poles occur at
points $s=\rho-iy$ for which $\rho=\beta+i\gamma$ is a zero of the zeta
function with $|\gamma-y|<T$. Arguing as in the proof of
Lemma~\ref{lem:summing-roots}, for each integer $n\in[0,T]$
there are at most $O(\euL)$ zeros
$\rho=\frac12+i\gamma$ of $\zeta(s)$ such that $n\le|\gamma-y|\le n+1$;
hence the sum of the residues from such zeros is
$$
-\ssum{\rho=\beta+i\gamma\\|\gamma-y|<T}\frac{x^{\rho-iy}}{\rho-iy}
\ll x^{1/2}\sum_{0\le n\le T}\ssum{n\le|\gamma-y|\le n+1}\frac{1}{n+1}
\ll x^{1/2}\euL\ssum{0\le n\le T}\frac{1}{n+1}
\ll x^{1/2}\euL^2.
$$
Putting everything together, both theorems are proved for $k=1$.

\bigskip\noindent{\sc Case 2 $(k\ge 2)$.}
In this case, let $\sC$ be the contour
$$
\sigma_0-iT
~~\longrightarrow~~\sigma_0+iT
~~\longrightarrow~~ \tfrac12+\tfrac{1}{\log x}+iT
~~\longrightarrow~~ \tfrac12+\tfrac{1}{\log x}-iT
~~\longrightarrow~~ \sigma_0-iT.
$$
Since every number $s+iy=\sigma+i(T+y)$ with
$\sigma\in(\frac12+\tfrac{1}{\log x},\sigma_0)$ is contained
in the region $\cR_x$ defined in Proposition~\ref{prop:condl-bd-log-deriv-zeta},
and for such $s+iy$ one has
$$
\tau=|T+y|+2\ll x+|y|,
$$
by the proposition we get that
\be\label{eq:alpha,y,s-bd-XX}
\alpha(y,s)\ll\big(\log(x\tau)\log\tau\big)^k\ll\euL^{2k},
\qquad\euL\defeq\log(x+|y|).
\ee
Consequently,
\be\label{eq:horiz-up-XX}
\int_{\frac12+\frac1{\log x}+iT}^{\sigma_0+iT}
\alpha(y,s)\,\frac{x^s}{s}\,ds
\ll\frac{\euL^{2k}}{T}\int_{\frac12+\frac1{\log x}}^{\sigma_0}x^\sigma\,d\sigma
<\frac{x^{\sigma_0}\euL^{2k}}{T\log x}
\ll x^{1/2}\euL^{2k}.
\ee
Similarly,
\be\label{eq:horiz-down-XX}
\int_{\frac12+\frac1{\log x}-iT}^{\sigma_0-iT}
\alpha(y,s)\,\frac{x^s}{s}\,ds
\ll x^{1/2}\euL^{2k}.
\ee
On the vertical segment, $s=\frac12+\tfrac{1}{\log x}+it$ with $|t|<T$,
the number $s+iy$ again lies in $\cR_x$. For such $s+iy$ one has
$$
\tau=|t+y|+2\ll x+|y|,
$$
thus Proposition~\ref{prop:condl-bd-log-deriv-zeta} again
yields the uniform bound \eqref{eq:alpha,y,s-bd-XX} along the vertical segment.
Consequently,
\begin{equation}
\label{eq:vert-left-XX}
\int_{\frac12+\frac1{\log x}-iT}^{\frac12+\frac1{\log x}+iT}
\alpha(y,s)\,\frac{x^s}{s}\,ds
\ll x^{\frac12+\frac1{\log x}}\euL^{2k}
\int_{-T}^T\frac{dt}{|t|+1}
\ll x^{1/2}\euL^{2k+1}.
\end{equation}
Combining \eqref{eq:horiz-up-XX}, \eqref{eq:horiz-down-XX}, and
\eqref{eq:vert-left-XX} with our previous estimate \eqref{eq:rush},
we have
\be\label{eq:rush2-XX}
\sum_{n\le x}a_n(y)
=\frac{1}{2\pi i}
\oint_{\sC}\alpha(y,s)\,
\frac{x^s}{s}\,ds+O(x^{1/2}\euL^{2k+1}).
\ee
The result follows from Cauchy's theorem, since the only pole
of the integrand with $\sigma>\frac12$ occurs at the point $s=1-iy$.
Note that this pole does not lie inside $\sC$ when $|y|>T$, and in
that case the integral vanishes from \eqref{eq:rush2-XX}

\section{Proof of Theorems~\ref{thm:main-back} and~\ref{thm:main2-back}}

As in the previous section, we prove both theorems in parallel.

Let the integer $k\ge 2$ be fixed throughout,
and fix $\eps>0$. In what follows, any implied constants in the symbols
$O$ and $\ll$ may depend (where obvious) on $k$ and $\eps$
but are independent of other parameters.

We continue to use notation introduced in the previous section.
In particular, for the proof of Theorem~\ref{thm:main-back} we have
$\alpha(0,s)=(-1)^k\{\zeta'(s)/\zeta(s)\}^k$, and for the proof
of Theorem~\ref{thm:main2-back}, $\alpha(0,s)=(-1)^k\zeta^{(k)}(s)/\zeta(s)$.
In either case, the Laurent series development of $\alpha(0,s)$
at $s=1$ has the form
$$
\alpha(0,s)=f(s)+g(s),\qquad
f(s)\defeq\sum_{m=1}^k\frac{a_m}{(s-1)^m},\qquad
g(s)\defeq\sum_{n=0}^\infty b_n(s-1)^n.
$$
Both $f$ and $g$ continue meromorphically to the 
entire complex plane. Following the method of \cite{GGL}, let us denote
$$
\Psi(x,y)\defeq\sum_{n\le x}a_n(y)
$$
(hence $\Psi=\psi^k$ or $\psi_k$)
and
$$
R(x,y)\defeq\Psi(x,y)-\mathop{{\tt Res}}\limits_{w=1-iy}
\Big(\alpha(y,w)\frac{x^w}{w}\Big).
$$
Our hypothesis (in both theorems) is that
\be\label{eq:hypothesis}
R(x,y)\ll x^{1/2}(x+|y|)^\eps.
\ee
Note that
$$
R(x,y)=\Psi(x,y)-
\mathop{{\tt Res}}\limits_{w=1}
\Big(f(w)\frac{x^{w-iy}}{w-iy}\Big)
$$
since $g$ is analytic in a neighborhood of one.

Let $H$ be the meromorphic function defined in
half-plane $\sigma>2$ by
\dalign{
H(s)&\defeq\int_1^\infty R(x,y)x^{-s}\,dx
=\int_1^\infty\Psi(x,y)x^{-s}\,dx-
\int_1^\infty\mathop{{\tt Res}}\limits_{w=1}
\Big(f(w)\frac{x^{w-iy}}{w-iy}\Big)\,x^{-s}\,dx.
}
In the same half-plane, we have
\dalign{
\int_1^\infty\Psi(x,y)x^{-s}\,dx
&=\int_1^\infty\bigg(\sum_{n\le x}a_n(y)\bigg)x^{-s}\,dx
=\sum_{n=1}^\infty a_n(y)\int_n^\infty x^{-s}\,dx\\
&=\sum_{n=1}^\infty a_n(y)\cdot\frac{n^{1-s}}{s-1}
=\frac{1}{s-1}\,\alpha(y,s-1).
}
Moreover, using Lemma~\ref{lem:saloni} we have
\dalign{
\int_1^\infty\mathop{{\tt Res}}\limits_{w=1}
\Big(f(w)\frac{x^{w-iy}}{w-iy}\Big)x^{-s}\,dx
&=\sum_{m=1}^ka_m\int_1^\infty\mathop{{\tt Res}}\limits_{w=1}
\bigg(\frac{1}{(w-1)^m}\frac{x^{w-iy}}{w-iy}\bigg)x^{-s}\,dx\\
&=\frac{1}{s-1}\sum_{m=1}^ka_m
\bigg(\frac{1}{(s-2+iy)^m}-\frac{1}{(iy-1)^m}\bigg)\\
&=\frac{1}{s-1}\big(f(s-1+iy)+C\big),
}
where $C$ is a constant that depends only on $k$ and $y$.
Therefore, 
$$
H(s)=\frac{1}{s-1}\big(\alpha(y,s-1)-f(s-1+iy)-C\big).
$$
For the sake of clarity, we adopt the following notation:
$$
\hat\sigma\defeq\sigma-1,\quad
\hat t\defeq y+t,\quad
\hat s\defeq \hat\sigma+i\hat t=s-1+iy,\quad
K(\hat s)\defeq H(\hat s+1-iy)=H(s).
$$
Then, throughout the half-plane $\hat\sigma>1$ we have
$$
K(\hat s)
=\frac{1}{\hat s-iy}\big(\alpha(0,\hat s)-f(\hat s)-C\big)
=\frac{1}{\hat s-iy}\big(g(\hat s)-C\big),
$$
and so it is clear that $K(\hat s)$ has meromorphic continuation to the
entire complex plane. Taking into account that $g(\hat s)$ is analytic
at $\hat s=1$, and that the only pole of $f(\hat s)$ 
occurs at $\hat s=1$, we see that the only poles of $K(\hat s)$
in the half-plane $\hat\sigma>0$ occur at zeros of the zeta function.
For any such zero $\rho$, let
$$
k_\star(\rho) \defeq\text{the order of 
the pole of $K(\hat s)$ at $\hat s=\rho$}.
$$
If $\alpha(0,\hat s)=(-1)^k\{\zeta'(\hat s)/\zeta(\hat s)\}^k$
(Theorem~\ref{thm:main-back}), then $k_\star(\rho)=k$ for every zero~$\rho$
of the zeta function.
When $\alpha(0,s)=(-1)^k\zeta^{(k)}(\hat s)/\zeta(\hat s)$
(Theorem~\ref{thm:main2-back}), each zero~$\rho$ gives rise to a pole
of $K(\hat s)$ of order $k_\star(\rho)\le k$. These statements are easily
verified by examining the Taylor series development of $\zeta(\hat s)$ at the
point $\hat s=\rho$.

Our aim is to show that \eqref{eq:hypothesis} implies RH.
Suppose, on the contrary, that \eqref{eq:hypothesis} holds, and 
$\rho_0=\beta_0+i\gamma_0$ is a zero of the zeta function with $\beta_0>\frac12$.
Let $m$ be the multiplicity of $\rho_0$, and define
$$
\kappa(\hat s)\defeq\frac{(\hat s-1)^k\zeta(\hat s)^k}
{(\hat s-\rho_0)^{mk-k_\star(\rho_0)+1}(\hat s+2)^{4k}},\qquad
h(s)\defeq\kappa(s-1+iy)=\kappa(\hat s).
$$
The function $\kappa(\hat s)$ is meromorphic for $\hat s\in\C$
and has been crafted so that (among other things) the product
$\kappa(\hat s)K(\hat s)$ has no pole in the half-plane $\hat\sigma>0$
other than a simple pole at $\hat s=\rho_0$.

We begin by observing that
$$
\frac{1}{2\pi i}\int_{3-i\infty}^{3+i\infty}h(s)H(s)\er^{s\log x}\,ds
=
\frac{1}{2\pi i}\int_{2-i\infty}^{2+i\infty}
\kappa(\hat s)K(\hat s)\er^{(\hat s+1-iy)\log x}\,d\hat s.
$$
In the latter integral, we shift the line of integration left
to the line $\hat\sigma=\tfrac14$, passing only the pole
at $\hat s=\rho_0$; the residue at that point is $c\,x^{\rho_0-iy+1}$,
where
$$
c\defeq \frac{(-1)^k (\rho_0-1)^k(\zeta^{(m)}(\rho_0))^k}{((m-1)!)^k(\rho_0-iy)(\rho_0+2)^{4k}}\quad(\text{Theorem~\ref{thm:main-back}})$$
or
$$c\defeq \frac{(-1)^k {k \choose k_*(\rho_0)} m(m-1)\dots (m-k_*(\rho_0)+1)\zeta^{(m+k-k_*(\rho_0))}(\rho_0)(\zeta^{(m)}(\rho_0))^{k-1}}{(m!)^{k-1}(m+k-k_*(\rho_0))!(\rho_0-iy)(\rho_0+2)^{4k}}\quad(\text{Theorem~\ref{thm:main2-back}}).
$$
Applying the bounds\footnote{More precisely, one has $\zeta^{(j)}(\tfrac14+iv)\ll(1+|v|)^{1/2}(\log(|v|+2))^j$
for fixed $j=0,1,\ldots,k$.}
$$
\zeta^{(j)}(\tfrac14+iv)\ll 1+|v|
\quad\text{and}\quad
f(\tfrac14+iv)\ll 1
\qquad(v\in\R,~j=0,1,\ldots,k),
$$
we derive that
\dalign{
&\int_{1/4-i\infty}^{1/4+i\infty}
\kappa(\hat s)K(\hat s)\er^{(\hat s+1-iy)\log x}\,d\hat s
\ll x^{5/4}\int_{-\infty}^\infty
\Big|\kappa(\tfrac14+iv)K(\tfrac14+iv)\Big|\,dv\\
&\qquad\qquad\ll x^{5/4}\int_{-\infty}^\infty
\frac{(1+|v|)^{2k}\,dv}
{(1+|v-\gamma_0|)^{mk-k_\star(\rho_0)+1}(1+|v-y|)(1+|v|)^{4k}}\ll x^{5/4}.
}
Consequently,
\be\label{eq:one}
\frac{1}{2\pi i}\int_{3-i\infty}^{3+i\infty}h(s)H(s)\er^{s\log x}\,ds
=c\,x^{\rho_0-iy+1}+O(x^{5/4}).
\ee

Next, we evaluate the left side of \eqref{eq:one} in a different way.
Put
\be\label{eq:w(u)-defd}
w(u)\defeq\frac{1}{2\pi i}\int_{3-i\infty}^{3+i\infty}h(s)\er^{us}\,ds.
\ee
We have
$$
h(s)\er^{us}=\frac{(s+iy-2)^k\zeta(s-1+iy)^k\er^{us}}
{(s-1+iy-\rho_0)^{mk-k_\star(\rho_0)+1}(s+iy+1)^{4k}},
$$
When $u<0$, we use the uniform bound
$$
h(s)\er^{us}\ll\frac{\er^{u\sigma}}
{(1+|y+t-\gamma_0|)^{mk-k_\star(\rho_0)+1}(1+|y+t|)^{3k}}
\qquad(\sigma\ge 3).
$$
Shifting the line of integration in the integral
\eqref{eq:w(u)-defd} right to $+\infty$, we conclude that $w(u)=0$.
On other hand, when $u\ge 0$ we shift the line of integration left
to the line $\sigma=-\frac54$. We pass a pole of order $4k$ at
the point $s=-1-iy$, which contributes a residue of size $O(1)$.
Using the bound (cf.~\cite[Cor.~10.5]{MontVau})
$$
|\zeta(s-1+iy)|\ll(1+|y+t|)^{11/4}\qquad(\sigma=-\tfrac54),
$$
it follows that
$$
h(s)\er^{us}\ll\frac{\er^{-5u/4}}
{(1+|y+t-\gamma_0|)^{mk-k_\star(\rho_0)+1} (1+|y+t|)^{k/4}}
$$
on the line $\sigma=-\frac54$, so the integral on the new line
converges absolutely. To summarize, we have shown that
\be\label{eq:wuhan}
w(u)=\begin{cases}
0&\quad\hbox{if $u<0$},\\
O(1)&\quad\hbox{if $u\ge 0$}.
\end{cases}
\ee
Now,
\dalign{
\frac{1}{2\pi i}\int_{3-i\infty}^{3+i\infty}h(s)H(s)\er^{s\log x}\,ds
&=\frac{1}{2\pi i}\int_{3-i\infty}^{3+i\infty}h(s)
\bigg(\int_1^\infty R(z,y)z^{-s}\,dz\bigg)\er^{s\log x}\,ds\\
&=\int_1^\infty R(z,y)
\bigg(\frac{1}{2\pi i}\int_{3-i\infty}^{3+i\infty}h(s)
\er^{s(\log x-\log z)}\,ds\bigg)\,dz\\
&=\int_1^\infty R(z,y)w(\log x-\log z)\,dz\ll\int_1^x R(z,y)\,dz,
}
where we used \eqref{eq:wuhan} in the last step. 

Combining the previous bound with \eqref{eq:one} we see that
$$
x^{\beta_0+1}\ll \big|c\,x^{\rho_0-iy+1}\big|\ll\int_1^x R(z,y)\,dz+x^{5/4},
$$
where $\beta_0$ is the real part of $\rho_0$.
Recalling \eqref{eq:hypothesis} we have
$$
\int_1^x R(z,y)\,dz
\ll\int_1^x z^{1/2}(z+|y|)^\eps\,dz
\ll x^{3/2}(x+|y|)^\eps,
$$
and therefore
$$
x^{\beta_0+1}\ll x^{3/2}(x+|y|)^\eps
$$
for every $\eps>0$. Since $\beta_0>\frac12$, this leads to the desired contradiction, 
concluding the proof of the theorems.

\end{document}